\documentclass[11pt]{amsart}

\usepackage{amsmath,amssymb,amsthm}
\usepackage{enumitem}
\usepackage{verbatim}

\usepackage{tikz}
\usetikzlibrary{matrix,arrows}

\usepackage[top=3cm,bottom=3cm,left=3cm,right=3cm]{geometry}
\usepackage{setspace}
\theoremstyle{plain}
\newtheorem{theorem}{Theorem}
\newtheorem{proposition}[theorem]{Proposition}
\newtheorem{lemma}[theorem]{Lemma}
\newtheorem{corollary}[theorem]{Corollary}
\newtheorem{fact}[theorem]{Fact}
\newtheorem*{claim}{Claim}

\theoremstyle{definition}

\newcommand{\onto}{\twoheadrightarrow}

\newcommand{\gen}[1]{\langle #1 \rangle}

\renewcommand{\theta}{\vartheta}
\renewcommand{\phi}{\varphi}

\newcommand{\emp}{\varepsilon}

\newcommand{\op}{\mathrm{op}}

\newcommand{\im}{\mathrm{im}}

\newcommand{\var}{\mathbf}
\newcommand{\SL}{\var{SL}}

\newcommand{\V}{\mathbf{V}}
\newcommand{\W}{\mathbf{W}}

\newcommand{\Sat}{\mathrm{Sat}}
\newcommand{\AP}{\var{A}}
\newcommand{\PL}{\mathrm{PL}}

\newcommand{\dsd}{\,{\ast\ast}\,}

\newcommand\pfun{\mathrel{\ooalign{\hfil$\mapstochar\mkern5mu$\hfil\cr$\to$\cr}}}
\newcommand{\relto}{\pfun}

\numberwithin{theorem}{section}

\newif\ifshowproofs
\showproofsfalse

\title{Merge decompositions, two-sided Krohn-Rhodes, and aperiodic pointlikes}
\author{Samuel J. v. Gool \and Benjamin Steinberg}
\address{Department of Mathematics\\
    City College of New York\\
    Convent Avenue at 138th Street\\
    New York, New York 10031\\
    USA}
\email{samvangool@me.com\and bsteinberg@ccny.cuny.edu}
\thanks{The first-named author was supported by the European Union's Horizon 2020 research and innovation programme under the Marie Sklodowska-Curie grant \#655941; the second-named author was supported by United States - Israel Binational Science Foundation \#2012080 and NSA MSP \#H98230-16-1-0047.}

\date{\today}	

\begin{document}

\begin{abstract}
This paper provides short proofs of two fundamental theorems of finite semigroup theory whose previous proofs were significantly longer, namely the two-sided Krohn-Rhodes decomposition theorem and Henckell's aperiodic pointlike theorem, using a new algebraic technique that we call the merge decomposition.  A prototypical application of this technique decomposes a semigroup $T$ into a two-sided semidirect product whose components are built from two subsemigroups $T_1,T_2$, which together generate $T$, and the subsemigroup generated by their setwise product $T_1T_2$.  In this sense we decompose $T$ by merging the subsemigroups $T_1$ and $T_2$.   More generally, our technique merges semigroup homomorphisms from free semigroups.
\end{abstract}

\maketitle

\section*{Introduction}
Eilenberg's variety theorem~\cite{Eilenberg76} provides a dictionary between formal language theory and finite semigroup theory.  In particular, membership problems in certain Boolean algebras of regular languages (languages accepted by finite automata) are equivalent to membership problems in varieties of finite semigroups.  Other natural problems in language theory transform into questions about pointlikes with respect to a variety of finite semigroups, a notion introduced by Henckell and Rhodes~\cite{Hen1988}.  An important problem in language theory is the separation problem: given disjoint regular languages, determine whether they can be separated by a language from a given variety of regular languages.  The separation problem is equivalent to decidability of pointlike pairs~\cite{Almeida99}, which is strictly stronger than the membership problem~\cite{RS99,AS03}.   Decidability of pointlikes can be used to obtain decidability of membership problems of related varieties.  For instance, the second author showed, using the decidability of aperiodic pointlikes and Zelmanov's solution to the restricted Burnside problem, that the join of the variety of aperiodic semigroups with any variety of finite groups of bounded exponent has decidable membership problem, answering a question of Rhodes and Volkov~\cite{Slice}.

The first decidability result on pointlikes was Henckell's theorem on the decidability of aperiodic pointlikes~\cite{Hen1988}, which for a long time was considered one of the most difficult results in the subject. Henckell not only provided a decidability algorithm: he also gave an elegant structural description of the aperiodic pointlike sets that we call Henckell's formula.  Henckell's original proof idea is a variation on the holonomy proof ~\cite{holonomypdt} of the Krohn-Rhodes theorem~\cite{PDT} for  directly decomposing semigroups into wreath products.  
The difficult part of Henckell's proof is to prove that a certain semigroup is aperiodic, which he does by wreath product embeddings.  In~\cite{HRS2010AP}, Henckell, Rhodes and the second author provided a direct proof that Henckell's semigroup is aperiodic, leading to a simpler and shorter proof of his main theorem.  They also extended the theorem beyond aperiodic pointlikes to the variety of semigroups whose subgroups have prime divisors belonging to a fixed set $\pi$ of primes (the restriction of this proof to the aperiodic case can be found in~\cite[Ch.~4]{RS2009}).  Although simpler than the original proof of Henckell~\cite{Hen1988}, the proof in~\cite{HRS2010AP} is still non-trivial.

Recently, Place and Zeitoun~\cite{PlaZei2016FO} gave a new proof of the decidability of aperiodic pointlikes, which, unlike the previous proofs, is inductive.  They use a language theoretic reformulation of the problem of computing pointlike sets and the McNaughton-Sch\"utzenberger theorem that the aperiodic languages are precisely the first order definable languages~\cite{Str1994}. 
The Place-Zeitoun approach follows the inductive proof scheme of the Krohn-Rhodes theorem (the so-called `$V\cup T$' argument~\cite{Arbib,RS2009}) later used by Wilke in the logic context~\cite{Wil99}, but done in the power set of the semigroup.

This paper introduces a new algebraic tool, that we call the merge decomposition, in Section~\ref{sec:merge}. 
In Section~\ref{sec:KR}, we use this tool to give a short proof of the inductive step of the two-sided Krohn-Rhodes decomposition theorem (cf.~\cite[Ch.~5]{RS2009}). Then, in Section~\ref{sec:pointlikes}, we use the merge decomposition in the inductive step of the Place-Zeitoun inductive scheme to give a short algebraic proof of Henckell's formula for the aperiodic pointlikes.  We feel that our approach has several advantages over previous approaches~\cite{Hen1988,HRS2010AP,PlaZei2016FO}.  First of all, it leads to a significantly shorter proof than the previous ones.   Secondly, we obtain the best known bound on the length of a two-sided Krohn-Rhodes decomposition of the aperiodic semigroup witnessing pointlikes (or, equivalently, quantifier-depth of the first order formula giving separation).

 An advantage of our approach is that it is potentially extendable beyond the realm of first order logic on words.  For instance, decidability of pointlikes for some larger varieties than aperiodics is obtained in~\cite{HRS2010AP}.  We leave this to future work.



\section{Preliminaries}\label{sec:prelim}
We assume familiarity with notions from the theory of semigroups, in particular, relational morphisms and divisions, the wreath product (denoted $\wr$), and the two-sided semidirect product of semigroups (denoted $\bowtie$) and of varieties of finite semigroups (denoted $\ast\ast$); see, e.g.,~\cite[Ch.~1]{RS2009}. Throughout the paper, we call `variety' what is called `{\it pseudo}variety' in~\cite{RS2009}.

\subsubsection*{Augmented semigroups}
Let $T$ be a finite semigroup. Let $T^I$ be the monoid obtained by adjoining a new identity, $I$, to $T$ and $T^0$ the semigroup obtained by adjoining a new zero to $T$.
We denote by $\SL$ the variety of finite semilattices and by $U_1$ the two-element semilattice.
\begin{fact}\label{fac:I}
If a variety $\V$ contains $T$ and $\SL$, then $T^0 \in \V$. If a variety $\V$ contains $T$, $\SL$, and is generated by monoids, then $T^I \in \V$.
\end{fact}
\begin{proof}
The first statement is true because $T^0$ is a homomorphic image of $T \times U_1$~\cite[Ex.~I.9.2]{Eilenberg76}. For the second statement, we distinguish two cases. If $T$ is a monoid, then $T^I$ embeds in $T \times U_1$, where $U_1$ denotes the two-element semilattice~\cite[Ex.~I.9.1]{Eilenberg76}. If $T$ is not a monoid, then $T$ divides some monoid $M \in \V$, and since $T$ is not a monoid, it follows that $T^I$ divides the same monoid $M$.
\end{proof}

The semigroup $T$ acts faithfully on $T^I$ by multiplication on the right, and thus $T$ embeds into the semigroup of total functions on $T^I$; we identify every element $t \in T$ with the corresponding right multiplication map. Further, for every $t \in T^I$, we denote by $t^\sharp$ the function with constant value $t$. We define $T^\sharp := T \cup \{t^\sharp \ : \ t \in T^I\}$, the semigroup consisting of the right multiplication maps and the constant maps.\ Thus, $T^\sharp$ naturally acts on $T$ on the right.\footnote{Note that our definition of $T^\sharp$ for a semigroup $T$ deviates slightly from the definition of $M^\sharp$ for a monoid $M$ in~\cite[Subsec.~4.1.2]{RS2009}.}  Dually, $T^\flat$ denotes the semigroup consisting of left multiplication maps for every $t \in T$ and constant maps $t^\flat$ for every $t \in T^I$. Note that $T^\flat = ((T^\op)^\sharp)^\op$, and $T^\flat$ acts on $T$ on the left.

\begin{fact}\label{fac:sharp}
For any finite semigroup $T$, let $\widetilde{T}$ denote the monoid obtained from $T$ by adjoining an identity and a zero and let $M$ be any monoid with $|M| > |T|$. Then $T^\flat$ embeds in $M \wr \widetilde{T}$.
\end{fact}
\begin{proof}
Fix a bijection $t \mapsto m_t$ between $T$ and a subset of $M \setminus \{1_M\}$.
We define a function $i \colon T^\flat \to M \wr \widetilde{T}$.
For every $t \in T$, define $i(t) := (c_{1}, t)$, where $c_{1} \in M^{\widetilde{T}}$ denotes the function with constant value $1_M$, the identity of $M$, and $i(t^\flat) := (f_t, 0)$, where $f_t \in M^{\widetilde{T}}$ denotes the function defined by $f_t(0) := 1_M$ and $f_t(t') := m_{t't}$ for all $t' \in T^I$. It is straightforward to verify that $i$ is an injective homomorphism.
\end{proof}

\subsubsection*{Triple product}
Let $(S,+)$ be a (not necessarily commutative) semigroup equipped with two actions on it, a left action of a semigroup $(S_L,\cdot)$ and a right action of a semigroup $(S_R,\cdot)$, which commute. The \emph{triple product}\footnote{We follow the notation of~\cite[Sec V.9]{Eilenberg76}; note the positions of the semigroups acting on the left and on the right. Also note that the multiplication can be viewed as matrix multiplication, if we represent an element $(s_R, s, s_L)$ by the lower triangular matrix $\bigl[ \begin{smallmatrix}s_R & 0\\ s & s_L\end{smallmatrix}\bigr]$.}  $T = (S_R,S,S_L)$ is the semigroup of triples $(s_R, s, s_L)$, with multiplication defined by $(s_R, s, s_L) \cdot (s_R',s',s_L') := (s_Rs_R',ss_R'+s_Ls', s_Ls_L')$.

\begin{fact}\label{fac:triple}
If $S \in \V$ and $S_L, S_R \in \W$, then $(S_R,S,S_L) \in \V \dsd \W$.
\end{fact}
\begin{proof}
Define an action of $S_L \times S_R$ on $S$ by $\langle s_L,s_R \rangle s := s_Ls$ and $s \langle s_L,s_R \rangle := ss_R$. Then $T$ is isomorphic to the two-sided semidirect product $S \bowtie (S_L \times S_R)$. (Cf., e.g.,~\cite[Sec.~V.9]{Eilenberg76}.)
\end{proof}

\section{The merge decomposition}\label{sec:merge}
Throughout this section, we fix:
\begin{itemize}
	\item a finite alphabet $A$ and two disjoint subalphabets $A_1$, $A_2$ such that $A = A_1 \cup A_2$;
	\item two homomorphisms $\psi_1 \colon A_1^+ \to T_1$ and $\psi_2 \colon A_2^+ \to T_2$, with $T_1$ and $T_2$ finite;
	\item a homomorphism $\chi \colon (T_1 \times T_2)^+ \to T_0$.
\end{itemize}
For any $w_1 \in A_1^+$, $w_2 \in A_2^+$, define $\mu(w_1w_2) := (\psi_1(w_1),\psi_2(w_2)) \in T_1 \times T_2$. Since the subsemigroup $(A_1^+A_2^+)^+$ of $A^+$ is freely generated by the infinite set of generators $A_1^+A_2^+$, the function $\mu$ extends uniquely to a homomorphism $\mu \colon (A_1^+A_2^+)^+ \to (T_1 \times T_2)^+$. We define $\psi_0 \colon (A_1^+A_2^+)^+ \to T_0$ to be the composition $\chi \circ \mu$.
For $i = 0, 1, 2$, we denote the external identity of $T_i^I$ by $I_i$, and we also denote by $\psi_i$ the homomorphism from the corresponding free \emph{monoid} to the finite monoid $T_i^I$; i.e., $\psi_i(\emp) := I_i$.

For any word $w$ in $A^+$, uniquely write $w = v_2 u v_1$, with $v_2 \in A_2^*$, $u \in (A_1^+A_2^+)^*$, and $v_1 \in A_1^*$, and define $\tau(w) := (\psi_2(v_2),\psi_0(u),\psi_1(v_1)).$ The function $\tau \colon A^+ \to T_2^I \times T_0^I \times T_1^I$ is not a homomorphism in general. The aim in this section is to show that the kernel of $\tau$ can be refined to a semigroup congruence of  finite index in a well-controlled variety.

To this end, we will define a semigroup $T_M$ and a homomorphism $\psi_M \colon A^+ \to T_M$. Let $S:= (T_0^I)^{T_1^I \times T_2^I}$, with the pointwise product of $T_0^I$, written additively. We define a left action of $T_1^\sharp$ and a right action of $T_2^\flat$ on $S$. For $s \in S$, $s_L \in T_1^\sharp$ and $s_R \in T_2^\flat$, let $s_L s s_R \in S$ be defined by $[s_L s s_R](t_1,t_2) := s(t_1s_L,s_Rt_2)$ for every $(t_1,t_2) \in T_1^I \times T_2^I$. Let $T_M := (T_2^\flat,S,T_1^\sharp)$ be the triple product; we call $T_M$ the \emph{merge semigroup} associated to $\psi_1$, $\psi_2$ and $\chi$.
\begin{fact}\label{fac:merge}
Let $\V$ be a variety, and $\W$ a variety generated by monoids and containing $\SL$. If $T_0 \in \V$ , $T_1, T_2 \in \W$, and $T_M$ is any triple product of $T_2^\flat$, $(T_0^I)^{T_1^I \times T_2^I}$, and $T_1^\sharp$, then $T_M \in \V \dsd (\SL \dsd \W)$.
\end{fact}
\begin{proof}
Applying Fact~\ref{fac:sharp} with $M$ a semilattice (e.g., a chain) with $|T_2| + 1$ elements and using $\widetilde{T_2} \in \W$ by Fact~\ref{fac:I} yields $T_2^\flat \in \SL \dsd \W$. Similarly, $T_1^\sharp \in \SL \dsd \W$. Fact~\ref{fac:triple} gives the result.
\end{proof}
For any $w_1 \in A_1^+$, we define an element $s_{w_1} \in S$ by $s_{w_1}(t_1,I_2) := I_0$ and $s_{w_1}(t_1,t_2) := \chi(t_1\psi_1(w_1),t_2)$, for all $t_1 \in T_1^I$ and $t_2 \in T_2$.
Now let $\psi_M \colon A^+ \to T_M$ be the unique homomorphism defined by
\[
\psi_M(a_1) :=
(I_2^\flat, s_{a_1}, \psi_1(a_1))
\text{ for } a_1 \in A_1,
\quad
\psi_M(a_2) :=
(\psi_2(a_2), i_0, I_1^\sharp)
\text{ for } a_2 \in A_2,
\]
where $i_0$ denotes the identity of $S$, i.e., the function with constant value $I_0$. We call the homomorphism $\psi_M \colon A^+ \to T_M$ the \emph{merge decomposition} of $A^+$ along $\chi$, $\psi_1$ and $\psi_2$.

The crucial property of the merge decomposition is the following.
\begin{proposition}\label{prop:merge}
There exists a function $f \colon T_M \to T_2^I \times T_0^I \times T_1^I$ such that $f \circ \psi_M = \tau$.
\end{proposition}
\begin{proof}
For any $(t_2,s,t_1) \in T_M$, define $f(t_2,s,t_1) := (t_2 I_2, s(I_1,I_2), I_1 t_1)$. We show $f \circ \psi_M = \tau$.

We first prove, for all $w_1 \in A_1^+$,
$\psi_M(w_1) = (I_2^\flat, s_{w_1}, \psi_1(w_1)).$
By induction, assume that this holds for all shorter words in $A_1^+$. Then, writing $w_1 = a_1w_1'$, the left and right coordinates are clearly as stated, and the middle coordinate of $\psi_M(w_1) = \psi_M(a_1)\psi_M(w_1')$ is $s_{a_1}I_2^\flat + \psi_1(a_1)s_{w_1'}$. From the definition of the right action and of $s_{a_1}$ we get that $s_{a_1}I_2^\flat = i_0$. From the definition of the left action and of $s_{w_1'}$ and $s_{w_1}$, we get that $\psi_1(a_1)s_{w_1'} = s_{w_1}$.
For $w_2 \in A_2^+$, we easily obtain
$\psi_M(w_2) = (\psi_2(w_2), i_0, I_1^\sharp),$
since $s_L i_0 s_R = i_0$ for all $s_L$, $s_R$, because $i_0$ is a constant map.
Multiplying these two results, for any $w_1 \in A_1^+$ and $w_2 \in A_2^+$,
$
\psi_M(w_1w_2) = (I_2^\flat, s_{w_1w_2}, I_1^\sharp),$
where
$s_{w_1w_2}(I_1,I_2) = s_{w_1}(I_1,\psi_2(w_2)) = \chi(\psi_1(w_1),\psi_2(w_2)) = \psi_0(w_1w_2).$

We next prove, by induction on the length of $u \in (A_1^+A_2^+)^+$ as a word in the free semigroup generated by $A_1^+A_2^+$, that
$
\psi_M(u) =  (I_2^\flat, s_u, I_1^\sharp),$
where $s_u(I_1,I_2) = \psi_0(u)$. We have already established the base case. If $u = (w_1w_2)u'$ for some $w_1 \in A_1^+$ and $w_2 \in A_2^+$ with $u' \in (A_1^+A_2^+)^+$, then, for the middle coordinate $s_u$ of $\psi_M(u) = \psi_M(w_1w_2)\psi_M(u')$, we have
\[ s_u(I_1,I_2) = [s_{w_1w_2}I_2^\flat + I_1^\sharp s_{u'}](I_1,I_2) = \psi_0(w_1w_2) \cdot \psi_0(u') = \psi_0(u).\]

Finally, to prove that $f \circ \psi_M = \tau$, let $w \in A^+$. Suppose that $w = v_2 u v_1$ with $u \in (A_1^+A_2^+)^+$, $v_1 \in A_1^+$ and $v_2 \in A_2^+$. Then, using our previous calculations, we get
\[
\psi_M(v_2uv_1) = (\psi_2(v_2),i_0, I_1^\sharp) \cdot (I_2^\flat, s_u, I_1^\sharp) \cdot (I_2^\flat, s_{v_1},\psi_1(v_1)) = (\psi_2(v_2)^\flat, s, \psi_1(v_1)^\sharp),
\]
where
\[ s(I_1,I_2) = \left( \left(i_0 I_2^\flat + I_1^\sharp s_u\right) I_2^\flat + I_1^\sharp s_{v_1}\right)(I_1,I_2) = I_0 \cdot s_u(I_1,I_2) \cdot I_0 = \psi_0(u).\]
Thus, in this case, $f(\psi_M(w)) = \tau(w)$. If one or more of the factors in the factorization $w = v_2 u v_1$ are empty, then the proof is similar but simpler.
\end{proof}

We end with a prototypical application of the technique, to be used in the next section.
\begin{corollary}\label{cor:merge}
Let $S$ be a finite semigroup and let $T_1, T_2$ be subsemigroups of $S$ such that $T_1 \cup T_2$ generates $S$. Denote by $T_0 := \gen{T_1T_2}$, the subsemigroup generated by $T_1T_2$. Then the semigroup $S$ divides a triple product of $T_2^\flat$, $(T_0^I)^{T_1^I \times T_2^I}$, and $T_1^\sharp$.
\end{corollary}
\begin{proof}
Let $A_i := T_i \times \{i\}$ for $i = 1,2$ and $A := A_1 \cup A_2$. Denote by $\psi \colon A^+ \onto S$ the surjective homomorphism defined on generators $(t_i,i) \in A$ by $\psi(t_i,i) := t_i$. For $i = 1,2$, let $\psi_i$ be the restriction of $\psi$ to $A_i^+$, and let $\chi \colon (T_1 \times T_2)^+ \to T_0$ be the homomorphism defined by $\chi(t_1,t_2) := t_1t_2$ for $(t_1,t_2) \in T_1 \times T_2$. Note that $\psi_0$, as defined above, in this case turns out to be the restriction of $\psi$ to $(A_1^+A_2^+)^+$. Hence, writing $m \colon T_2^I \times T_0^I \times T_1^I \to T^I$ for the multiplication map $m(t_2,t_0,t_1) := t_2t_0t_1$, we have $\psi = m \circ \tau$. Let $\psi_M \colon A^+ \to T_M$ be the merge decomposition along $\chi$, $\psi_1$, and $\psi_2$. By Proposition~\ref{prop:merge}, pick $f \colon T_M \to T_2^I \times T_0^I \times T_1^I$ such that $\tau = f \circ \psi_M$. Then $\psi = m \circ f \circ \psi_M$, so $S$ divides $T_M$ since $\psi$ is surjective.
\end{proof}

\section{Two-sided Krohn-Rhodes theorem}\label{sec:KR}
In this section, we apply the merge decomposition technique of Section~\ref{sec:merge} to give a short proof of the crucial step in the two-sided Krohn-Rhodes theorem.

For any finite semigroup $S$, define $\V_S$ to be the smallest variety which is closed under two-sided semidirect products, and which contains $\SL$ and all simple groups that divide $S$.
\begin{theorem}[Two-sided Krohn-Rhodes]\label{thm:2KR}
Let $S$ be a finite semigroup. Then $S \in \V_S$.
\end{theorem}
\begin{proof}
By induction on $|S|$.

\underline{Case 1. $S$ is a group.} Any finite group embeds in an iterated wreath product of its simple group divisors, cf., e.g.,~\cite[Cor.~4.1.6]{RS2009}.

\underline{Case 2. $S$ is cyclic.} Any finite cyclic semigroup divides an iterated wreath product of a subgroup and copies of $U_1$, cf., e.g.,~\cite[Cor.~4.1.28]{RS2009}.

\underline{Case 3. $S$ is not a group and $S$ is not cyclic.} Let $A$ be a minimal generating set for $S$ and note that $|A|\geq 2$. Since $S$ is not a group, without loss of generality, $S$ is not right simple (cf., e.g.,~\cite[Lem.~A.3.3]{RS2009}). Therefore, there exists $a \in A$ such that $aS \subsetneq S$. Let $A_1 := \{a\}$, $A_2 := A \setminus A_1$, $T_i := \gen{A_i}$ for $i = 1,2$, and $T_0 := \gen{T_1T_2}$.
By minimality of $A$, $T_1$ and $T_2$ are strictly contained in $S$. By the induction hypothesis, $T_i \in \V_{T_i}$, which is contained in $\V_S$, since any simple group dividing $T_i$ also divides $S$.
Moreover, $T_0 \subseteq aS$, so $T_0$ is also strictly contained in $S$. By the induction hypothesis again, $T_0 \in \V_{T_0} \subseteq \V_S$. Since $T_1 \cup T_2$ generates $S$, by Corollary~\ref{cor:merge}, $S$ divides a triple product of $T_2^\flat$, $(T_0^I)^{T_1^I \times T_2^I}$, and $T_1^\sharp$. Hence, by Fact~\ref{fac:merge}, $S \in \V_S \dsd (\V_S \dsd \V_S) = \V_S$.
\end{proof}

\section{Henckell's theorem on aperiodic pointlikes}\label{sec:pointlikes}
Recall that any element $s$ in a finite semigroup $S$ has a unique {\it {idempotent power}}, $s^\omega$. A semigroup $S$ is called \emph{aperiodic} if every subgroup of $S$ is trivial, or, equivalently, $s^\omega s = s^\omega$ for every $s \in S$.
For $k \geq 1$, define $\SL^{k+1} := \SL \dsd \SL^k$.
A semigroup $S$ is aperiodic if, and only if, $S \in \SL^k$ for some $k$; indeed, the necessity follows from Theorem~\ref{thm:2KR}.\footnote{A finite semigroup $S$ lies in $\SL^k$ if, and only if, every language recognized by $S$ can be defined by a first-order sentence of quantifier depth $\leq k$; this result is contained in~\cite[Ch. VI]{Str1994}, and relates our work in Section~\ref{sec:pointlikes} to the logical approach of~\cite{PlaZei2016FO}.}

\begin{fact}\label{fac:dsddepth}
For any $m, n \geq 1$, $\SL^m \dsd \SL^n \subseteq \SL^{m+n}$.
\end{fact}
\begin{proof}
By induction on $m$. The case $m = 1$ is true by definition. By the lax associativity of double semidirect product~\cite[Cor.~2.6.26]{RS2009}, $(\SL \dsd \SL^{m-1}) \dsd \SL^n \subseteq \SL \dsd (\SL^{m-1} \dsd \SL^n)$. By the induction hypothesis, $\SL \dsd (\SL^{m-1} \dsd \SL^n) \subseteq \SL \dsd \SL^{m + n - 1} = \SL^{m+n}$.
\end{proof}

Let $\V$ be a variety. A subset $X$ of a finite semigroup $S$ is called \emph{$\V$-pointlike} if, for any relational morphism $\rho \colon S \relto T$ with $T \in \V$, $X \subseteq \rho^{-1}(t)$ for some $t \in T$.
Any singleton set is $\V$-pointlike, and the collection of $\V$-pointlike subsets of a semigroup $S$ forms a downward closed subsemigroup, $\PL_{\V}(S)$, of the power semigroup $2^S$, partially ordered by inclusion, and with multiplication of subsets of $S$.

The following observation is specific to the variety $\AP$ of aperiodic semigroups: if $X$ is an $\AP$-pointlike set in $S$, then so is the set
$X^{\omega + \ast} := \bigcup_{n \geq 0} X^{\omega} X^n.$
Indeed, for any $\rho \colon S \relto T$ with $T$ aperiodic, $X \subseteq \rho^{-1}(t)$ for some $t \in T$, which gives $X^m \subseteq \rho^{-1}(t^m)$ for all $m \geq 1$. Aperiodicity of $T$ then yields $X^{\omega}X^n \subseteq \rho^{-1}(t^\omega)$ for all $n \geq 0$.

We will call a subset $U$ of $2^S$ \emph{saturated}, if it is a subsemigroup that is closed downward in the inclusion order and closed under the operation $X \mapsto X^{\omega + \ast}$. Clearly, any subset $U$ of $2^S$ is contained in a smallest saturated set, which we call its \emph{saturation}, and denote by $\Sat(U)$.

We will need the following lemma, which was essentially already in~\cite{Hen1988}; see also~\cite{HRS2010AP}.
\begin{lemma}\label{lem:groupsat}
Let $G$ be a subgroup of $2^S$. Then $\bigcup G \in \Sat(G)$.
\end{lemma}
\begin{proof}
Let $C_1, \dots, C_k$ be an exhaustive list of the cyclic subgroups of $G$. Note that, for any generator $X$ of $C_i$, $X^{\omega + \ast} = \bigcup C_i$, so $\bigcup C_i \in \Sat(G)$ for every $i$. Also note that $G = C_1 \cdots C_k$. Therefore, since multiplication distributes over union, $\bigcup G = \left(\bigcup C_1\right) \cdots \left(\bigcup C_k\right) \in \Sat(G)$.
\end{proof}

We will use the merge decomposition (Section~\ref{sec:merge}) to give a short proof of the following theorem.
\begin{theorem}[cf.~\cite{Hen1988,HRS2010AP,PlaZei2016FO}]\label{thm:APPL}
Let $S$ be a semigroup. The set $\PL_{\AP}(S)$ is the saturation of the set of singletons in $2^S$. Moreover, if $A$ is a generating set for $S$, then $\PL_{\AP}(S) = \PL_{\SL^k}(S)$, where $k = (|A|-1)2^{\binom{|S|}{2}} + 2^{|A|} - 1$.
\end{theorem}

\begin{proof}
Throughout the proof, for any finite alphabet $A$, semigroup $S$, and homomorphism $\phi \colon A^+ \to 2^S$, define $U_\phi := \im(\phi)$, $S_{\phi} := \bigcup U_\phi$, and $k(\phi) := (|\phi(A)|-1)2^{{}^{\binom{|S_{\phi}|}{2}}} + 2^{|S_{\phi}|} - 1$.

\begin{claim}\label{prop:APinduction}
For any homomorphism $\phi \colon A^+ \to 2^S\setminus \{\emptyset\}$, there exists a homomorphism $\psi \colon A^+ \to T$ with $T \in \SL^{k(\phi)}$ and $\bigcup \phi(\psi^{-1} t) \in \Sat(U_\phi)$ for every $t \in T$.
\end{claim}
\begin{proof}[Proof of Claim.]
The construction of $\psi \colon A^+ \to T$ with $T \in \SL^{k(\phi)}$ is by induction on the parameter $(|S_\phi|,|\phi(A)|)$ in $\mathbb{N}^2$, ordered lexicographically.\\[2mm]
\underline{Case 1. For every $a \in A$, $\phi(a)S_{\phi} = S_{\phi} = S_{\phi}\phi(a)$.}\\[1mm]
Let $e = \phi(w)$ be an idempotent in the minimal ideal of $U_\phi$. Then $G := e U_\phi e$ is a subgroup of $U_\phi$, see, e.g.,~\cite[App.~A]{RS2009}. By Lemma~\ref{lem:groupsat}, $\bigcup G$ lies in $\Sat(e U_\phi e)$, and hence also in $\Sat(U_\phi)$, since $e U_\phi e \subseteq U_\phi$. Using the assumption in this case and the fact that multiplication distributes over union, we have
\[ S_\phi = \phi(w) S_\phi \phi(w) = e \left(\bigcup U_\phi\right) e  = \bigcup G.\]
Thus, $S_\phi$ lies in $\Sat(U_\phi)$, and we choose $\psi$ to be the trivial homomorphism $A^+ \to \{1\} \in \SL$.\\[2mm]
\underline{Case 2. $|\phi(A)| = 1$.}\\[1mm]
Denote the unique element of $\phi(A)$ by $X$. Since $U_\phi$ is a finite cyclic semigroup, pick $m \leq |U_\phi|$ such that $X^m$ is idempotent, i.e., $X^m = X^\omega$. Let $T := \gen{x \ | \ x^m = x^{m+1}}$, the finite aperiodic cyclic semigroup of order $m$, and let $\psi \colon A^+ \to T$ be the homomorphism defined by $a \mapsto x$ for every letter $a \in A$. Note that $T \in \SL^m$~\cite[Lem.~4.1.27]{RS2009}, and, since $U_\phi \subseteq 2^{S_\phi} \setminus \{\emptyset\},$ we have $m \leq |U_\phi| \leq 2^{|S_\phi|} - 1 = k(\phi)$.
From the definitions, note that, for $1 \leq i < m$, $\bigcup \phi(\psi^{-1} x^i) = X^i$, which lies in $U_\phi$, and for $i \geq m$, $\bigcup \phi(\psi^{-1} x^i) = X^{\omega + \ast}$, which lies in $\Sat(U_\phi)$.\\[2mm]
\underline{Case 3. $|\phi(A)| \geq 2$, and there is $a_0 \in A$ such that $\phi(a_0)S_\phi \subsetneq S_\phi$ or $S_\phi \phi(a_0) \subsetneq S_\phi$.}\\[1mm]
 Without loss of generality, we may assume $\phi(a_0)S_\phi \subsetneq S_\phi$.
Let $A_1 := \{ a \in A \ | \ \phi(a) = \phi(a_0)\}$, and $A_2 := A \setminus A_1$. Note that, since $|\phi(A)| \geq 2$, $\phi(A_1)$ and $\phi(A_2)$ are non-empty proper subsets of $\phi(A)$. For $i = 1,2$, denote by $\phi_i$ the restriction of $\phi$ to $A_i^+$, and pick $\psi_i \colon A_i^+ \to T_i$ with $T_i \in \SL^{k(\phi_i)}$ and $\bigcup \phi(\psi_i^{-1} t) = \bigcup \phi_i(\psi_i^{-1} t) \in \Sat(U_{\phi_i}) \subseteq \Sat(U_\phi)$, for all $t \in T_i$.  Without loss of generality, we may assume the $\psi_i$ are surjective.

Let $\phi_0 \colon (T_1 \times T_2)^+ \to 2^S\setminus \{\emptyset\}$ be the unique homomorphism defined, for $(t_1,t_2) \in T_1 \times T_2$, by $\phi_0(t_1,t_2) := \bigcup \phi(\psi_1^{-1}t_1 \cdot \psi_2^{-1}t_2)$. Note that $S_{\phi_0} \subseteq \phi(a_0) S_\phi$, since any $w \in \psi_1^{-1}t_1 \cdot \psi_2^{-1}t_2$ starts with a letter from the subalphabet $A_1$. Since $\phi(a_0) S_\phi \subsetneq S_\phi$ by assumption, $|S_{\phi_0}| < |S_\phi|$, so the induction hypothesis applies to $\phi_0$: pick a homomorphism $\chi \colon (T_1 \times T_2)^+ \to T_0$ with $T_0 \in \SL^{k(\phi_0)}$ such that $\bigcup \phi_0(\chi^{-1}(t)) \in \Sat(U_{\phi_0}) \subseteq \Sat(U_\phi)$, for every $t \in T_0$.

Define $\mu \colon (A_1^+A_2^+)^+ \to (T_1 \times T_2)^+$ and $\psi_0 := \chi \circ \mu$, as in Section~\ref{sec:merge}.
Note that, for any $w_1 \in A_1^+$, $w_2 \in A_2^+$, we have $\phi(w_1w_2) \subseteq \phi_0(\mu(w_1w_2))$, and, hence, $\phi(w) \subseteq \phi_0(\mu(w))$ for all $w \in (A_1^+A_2^+)^+$. Therefore, by the definition of $\psi_0$, $\bigcup \phi(\psi_0^{-1}t) \subseteq \bigcup \phi_0(\chi^{-1}t)$ for all $t \in T_0$, so also $\bigcup \phi(\psi_0^{-1}t) \in \Sat(U_\phi)$.
Applying the construction of Section~\ref{sec:merge}, let $\psi_M \colon A^+ \to T_M$ be the merge homomorphism, and pick $f \colon T_M \to T_2^I \times T_0^I \times T_1^I$ such that $f \circ \psi_M = \tau$. Let $t \in T_M$, and write $f(t) = (t_2,t_0,t_1) \in T_2^I \times T_0^I \times T_1^I$. If $(t_2,t_0,t_1) \in T_2 \times T_0 \times T_1$, then
\[ \bigcup \phi(\psi_M^{-1}t) \subseteq \bigcup \phi(\tau^{-1}(t_2,t_0,t_1)) =
\bigcup \phi(\psi_2^{-1}t_2) \cdot \bigcup \phi(\psi_0^{-1}t_0) \cdot \bigcup \phi(\psi_1^{-1}t_1) \in \Sat(U_\phi),\]
and, if $t_i = I_i$ for one or more $i \in \{0,1,2\}$, a similar inclusion holds, omitting the corresponding factors $\bigcup \phi(\psi_i^{-1}t_i)$ from the final product.

Let us write $m := |S_{\phi}|$. Note that, since $S_{\phi_0}$ is strictly contained in $S_{\phi}$ and $\phi_0(T_1 \times T_2)$ is contained in $2^{S_{\phi_0}} \setminus \{\emptyset\}$, we have
\begin{equation*}\label{eq:T0bound}
k(\phi_0) \leq (2^{m-1}-2)2^{\binom{m-1}{2}} + 2^{m-1} - 1 = 2^{\binom{m}{2}} - 2^{\binom{m-1}{2}+1} + 2^{m-1}-1 \leq 2^{\binom{m}{2}} - 1,
\end{equation*}
using that $\binom{m}{2} = m-1 + \binom{m-1}{2}$ and $2^{m-1} \leq 2^{\binom{m-1}{2}+1}$.

By Facts~\ref{fac:merge}~and~\ref{fac:dsddepth}, $T_M \in \SL^k$, where $k = k(\phi_0) + \max\{k(\phi_1),k(\phi_2)\} + 1$. Using that $|\phi(A_i)| < |\phi(A)|$, we have
\begin{align*}
k(\phi_0) + \max\{k(\phi_1),k(\phi_2)\} + 1 \leq \left(2^{\binom{m}{2}} - 1\right) + \left((|\phi(A)| - 2) 2^{\binom{m}{2}} + 2^{m} - 1\right) + 1 = k(\phi). \quad \qedhere
\end{align*}
\end{proof}

Now, to prove the theorem, let $A$ be a generating set for $S$, define $\phi \colon A^+ \to 2^S$ by $\phi(a) := \{a\}$ for $a \in A$, and pick $\psi \colon A^+ \to T$ as in the claim. Then $U_\phi$ is the set of singletons, $|\phi(A)| = |A|$, and $S_\phi = S$, so that $k(\phi) = (|A|-1)2^{\binom{|S|}{2}}+2^{|S|} - 1=:k$. Define the relational morphism $\rho \colon S \relto T$ by $\rho^{-1}(t) := \bigcup \phi(\psi^{-1} t)$. Then, for any $\SL^k$-pointlike $X \subseteq S$, we have $X \subseteq \rho^{-1}(t)$ for some $t \in T$, and therefore, since $\rho^{-1}(t)$ lies in $\Sat(U_\phi)$ by the claim, so does $X$. We have proved that $\PL_{\SL^k}(S) \subseteq \Sat(U_\phi)$, while the remarks at the beginning of this section imply $\Sat(U_\phi) \subseteq \PL_{\AP}(S)$, which is clearly contained in $\PL_{\SL^k}(S)$, since $\SL^k \subseteq \AP$. Thus, $\PL_{\SL^k}(S) = \Sat(U_\phi) = \PL_{\AP}(S)$.
\end{proof}

\section*{Acknowledgements}
In an earlier version of the proof of Theorem~\ref{thm:APPL}, we proved the claim for $k(\phi) := |\phi(A)|2^{{}^{\binom{|S_\phi|}{2}}}$. We acknowledge the help of MathOverflow~\cite{SteinbergMO2017} for guiding us to the slightly better bound given in the paper.
\bibliographystyle{amsplain}
\bibliography{goolsteinberg2017}

\end{document}